\documentclass[11pt]{article}   	% use "amsart" instead of "article" for AMSLaTeX format

\usepackage{amsmath,amssymb,latexsym,amsthm,mathrsfs,color}
\usepackage{graphicx,authblk}
\usepackage{graphicx}				% Use pdf, png, jpg, or eps§ with pdflatex; use eps in DVI mode
\usepackage{amssymb}
\usepackage{authblk}
\newtheorem{theo}{Theorem}[section]
\newtheorem{pro}[theo]{Proposition}
\newtheorem{lem}[theo]{Lemma}

\newcommand{\ra}{\rightarrow}

\theoremstyle{definition}

\newtheorem{exa}{Example}[section]

\theoremstyle{remark}
\newtheorem{rem}[theo]{Remark}

\title{The hit-and-run version of top-to-random}
\author[1]{Samuel Boardman\thanks{Research  supported by NSF RTG grant  DMS-1645643. Email: stb89@cornell.edu}}
 \author[2]{
 Daniel Rudolf\thanks{Research partially supported by DFG project 389483880. Email: 
 daniel.rudolf@uni-goettingen.de
}}

 \author[1]{ Laurent Saloff-Coste\thanks{Research partially supported by NSF  grant DMS-1707589. Email: lps2@cornell.edu}}
  
  \affil[1]{Department of Mathematics, Cornell University, Ithaca, NY, 14853, USA}
   \affil[2]{Institut f\"ur Mathematische Stochastik\\  Georg-August-Universit\"at G\"ottingen,
 Goldschmidtstr. 3-5, 37077 G\"ottingen,
 Germany}  
\begin{document}
\maketitle

\begin{abstract}We study an example of a {\em hit-and-run} random walk on the symmetric group $\mathbf S_n$. Our starting point is the well understood {\em top-to-random} shuffle. In the hit-and-run version, at each {\em single step}, after picking the point of insertion, $j$, uniformly at random in $\{1,\dots,n\}$, the top card is inserted in the $j$-th position $k$ times in a row where $k$ is uniform in $\{0,1,\dots,j-1\}$. The question is, does this accelerate mixing significantly or not?  We show that, in $L^2$ and sup-norm, this accelerates mixing at most by a constant factor (independent of $n$). Analyzing this problem in total variation is an interesting open question. We show that, in general, hit-and-run random walks on finite groups have non-negative spectrum.
\end{abstract}

\section{Introduction}  Given a finite group and a generating $k$-tuple,  consider the simple random walk on $G$ associated to this $k$-tuple.
At each integer time, this random walk moves from the current position $X_n$ to $X_ng$ where $g$ is picked uniformly at random among the $k$ generators, independently of all previous steps.  To define the {\em hit-and-run walk} based on the same generating $k$-tuple, for any group element $g$, call $m_g$ the order (i.e., exponent) of $g$. At each step,  pick one of the $k$ generators uniformly at random, call it $g$, pick $\ell$ uniformly in $\{0,\dots, m_g-1\}$, and move to $X_ng^\ell$.  

This defines a natural variation on simple random walks which allows for long jumps when the orders of some of the  generators are relatively large. As often in the study of random walks on finite groups,  it is easier to think about the problem for a family of random walks on a sequence of finite groups whose sizes increase to infinity.

Two of the most basic questions one can ask concerning a family of ergodic random walks on some finite groups whose sizes increase to infinity are:  How long does the walk take to be approximately uniformly distributed? Does the cut-off phenomenon occur? that is, is there a rapid transition from being far from equilibrium to reaching approximate equilibrium?  See \cite{Ald,Diabook,Dia-cutoff} for introductions to these problems. In the context of hit-and-run random walks, the following additional question emerges: Does the hit-and-run version converge faster than the simple random walk version?, i.e., does the extra randomization help and if so, how much?

We study these problems in the case of the hit-and-run walk associated with one of the classic random walks on the symmetric group, top-to-random. See the Example \ref{exa-HRTR} below. We show that, if convergence is measured in $L^2$, the hit-and-run walk and the original top-to-random walk both take order $n\log n$ to converge. What exactly happens to the hit-and-run walk in total variation is left as  an open question but it seems plausible that, again, it take order $n\log n$ to converge, as top-to-random does \cite{Ald,Diabook}. We give an analysis of the Markov chain consisting in following a fixed single card. While studying this example and based on some numerical evidence, the first and last authors conjectured that
the hit-and-run top-to-random walk had only non-negative eigenvalues. The second author provided a proof of this fact, and more, based on earlier works on hit-and-run algorithms \cite{RUpos}: for any generating tuple on any finite group, the associated hit-and-run walk has non-negative spectrum. See Theorem \ref{th-pos} and Section \ref{sec-pos}.

\subsection{Random walks based on generating $k$-tuples}  Let $G$ be a finite group. 
For any generating $k$-tuple $S=(g_1,\dots, g_k)$,
let $\mu_S$ be    probability measure 
$$\mu_S=\frac{1}{k}\sum_{i=1}^k\delta_{g_i},\;\;\delta_{g}(h)=\left\{\begin{array}{l} 1\mbox{ if } h=g\\0 \mbox{ otherwise}.\end{array}\right.$$

The random walk on the group $G$ driven by the measure $\mu_S$  above or any probability measure $\mu$, for that matter, is the Markov chain with state space $G$ and Markov kernel
$$M(x,y)=\mu(x^{-1}y).$$ 
The uniform measure $u=u_G$ on $G$ is always invariant for such a Markov chain and it is useful to consider the (convolution) operator 
$$f\mapsto  Mf(x)=\sum_yM(x,y)f(y)$$
acting on $L^2(G)=L^2(G,u)$.  At any (discrete) time $t$, the iterated kernel $M^t(x,y)$ is given by the $t$-fold convolution $\mu^{(t)}$ of $\mu$ by itself
in the form $M^t(x,y)=\mu^{(t)}(x^{-1}y)$.
The adjoint  $M^*$ of $M$ satisfies $M=M^*$ if and only if $\mu$ is symmetric in the sense that
$\check{\mu}(x)=\mu(x^{-1} )=\mu(x)$. 
 
\begin{exa}  \label{basicexa}
 The following examples on the symmetric group $\mathbf S_n$ will be of particular interest to us.  See \cite{Ald,DSh,Diabook,DFP,DSCcompG, DBHyp,BN,SCRW}.
\begin{itemize} 
\item (Top-to-random)  $S=(\sigma_i)_1^n$ where $\sigma_i$ take the top card of the deck and place it in position $i$. In cycle notation, $\sigma_i=(i,i-1,\dots, 2,1)$.  The probability measure $\mu_S$ in this example is not symmetric.
\item (Random-to-random or random insertions) $S= (\sigma_{ij})_{1\le i, j\le n}$ (ordered lexicographically) where $\sigma_{ij}$ is  ``take the card in position $i$ and insert it in position $j$.'' In cycle notation, when $i<j$, $\sigma_{ij}=(j,j-1,\dots,i)$.  Note also $\sigma_{ij}=\sigma_{ji}^{-1}$ and $\sigma_{ii}=e$. The corresponding measure $\mu_S$ gives probability $1/n$ to the identity element $e$ and probability $1/n^2$ to each $\sigma_{ij}$, $i\neq j$ with the caveat that when $|j-i|=1$ , $\sigma_{ij}=\sigma_{ji}$ so that
the corresponding transposition $\tau=\sigma_{ij}=\sigma_{ji}$ actually has probability $2/n^2$.
\item (Random transposition) Take $S=(\tau_{ij})_{1\le i \le j\le n}$ where $\tau_{ij}$ is transpose the cards in positions $i$ and $j$ (i.e., $\tau_{ij}=(i,j)$). This tuple $S$ contains each true transposition $(i,j)$, $1\le i<j\le n$, twice, and also includes $n$ copies of the identity $(i,i)$, $1\le i\le n$. Equivalently, we can think of $(i,j)$ being picked uniformly independently at random from $\{1,\dots,n\}$ so that the probability measure $\mu_S$ give probability $1/n$ to the identity and probability $2/n^2 $ to any transposition.  
\end{itemize}
\end{exa}

All these examples are ergodic in the sense that the distribution at time $t$ of the associated Markov chain converges to the uniform distribution $u$ on $\mathbf S_n$.  
\subsection{Hit-and-run walks based on generating tuples}

We now consider the following modification of the measure $\mu_S$ associated with a fixed generating tuple $S=(s_1,\dots,s_k)$ which we call $q_S$.
For each $s_i\in S$, let $m_i$ be its order in $G$ (the smallest $m$ such that $s_i^m=e$). Define
\begin{equation}\label{def-HR} q_S=\frac{1}{k}\sum_{i=1}^k \frac{1}{m_i}\sum_{j=0}^{m_i-1}\delta_{s_i^j}.\end{equation}
To describe $q_S$ in words,  $q_S$ is the distribution of a random element in $G$ chosen as follows: Pick $i$ uniformly in $\{1,\dots,k\}$, pick $m$ uniformly in $\{0,\dots,m_i-1\}$,
output $s_i^m\in G$. This is reminiscent to the so-called hit-and-run algorithms, hence the name.

The question we want to address is whether or not the random walk driven by $q_S$ mixes faster than the random walk driven by $\mu_S$. Does taking a uniform step in the direction of the generator $s_i$, that is, along the one parameter subgroup $\{s_i^m: 0\le m\le m_i-1\}$, instead of just a single $s_i$-step, speeds-up convergence or not?

\begin{exa}[Our main example: hit-and-run for top-to-random] \label{exa-HRTR}
Top-to-random on $\mathbf S_n$ is obtained by considering the generating  $n$-tuple 
$$S= \{ (k,k-1,\dots, 2,1): k=1,\dots,n\}=\{\sigma_k: k=1,\dots,n\}$$ where $\sigma_k:= (k, k - 1,\dots,2, 1)$. The associated simple random walk measure is
$$\mu_S(\sigma) =\left\{
  \begin{array}{cl}
    \frac{1}{n}  & \quad \text{if } \sigma\in S, \\
    0  & \quad \text{otherwise.}
  \end{array}\right.  $$
The associated hit-and-run measure  is given by
\begin{equation}\label{HRTR}q(\sigma)=q_S(\sigma)=\frac{1}{n}\sum_{i=1}^n \frac{1}{i}\sum_{j=0}^{i-1}\delta_{\sigma_i^j}(\sigma).\end{equation}
This probability measure is symmetric and gives positive probability to order $n^2$ distinct permutations.

Let us now describe our findings and related open questions regarding the hit-and-run for top-to-random shuffle.
\begin{itemize}
\item (Facts) In $L^2$, the mixing time for hit-and-run for top-to-random with $n$ cards is of order $n\log n$, the same order than the top-to-random shuffle. See Section \ref{sec-HRTR2}.   There is a cut-off in $L^2$ but the cut-off time is not known.  In $L^1$ (i.e., total variation), the mixing time is at least of order $n$ and no more than order $n\log n$. 
\item  (Open questions) What is the cutoff time in $L^2$ for the hit-and-run version of top to random? 
How does it compare precisely with $n\log n$, the cutoff time for the top-to-random shuffle?
\item (Open questions)  Is there a cut-off in $L^1$ (i.e., total variation)? What is the order of magnitude of the $L^1$-mixing time? Describe a simple statistics
that provides a good lower bound for the mixing time in $L^1$.
\item (Conjecture)  There is a cut-off in $L^1$ and the rough order of the $L^1$ cut-off time is $n\log n$.
\end{itemize}
\end{exa}

Regarding general hit-and-run walks, we prove the following result.
\begin{theo} \label{th-pos}Let $G$ be a finite group and $S=(s_1,\dots,s_k)$ be a generating tuple. The eigenvalues $-1\le \beta_{|G|-1}\le \dots\le \beta_1\le \beta_0=1$ of the hit-and-run walk on $G$ based on $S$ driven by the symmetric measure $q_{S}$ at {\em (\ref{def-HR})} are all non-negative, that is
$0\le \beta_{|G|-1}\le \dots\le \beta_1\le \beta_0=1$.
\end{theo}

The proof of this theorem is in Section \ref{sec-pos}.  Section \ref{sec-one} provides exact computations concerning the Markov chains obtained by following a single card. We explore the time to equilibrium for this Markov chain as a function of the starting position of the card that is followed, both in total variation and in $L^2$.    Section \ref{sec-HRTR2} studies the convergence of the hit-and-run top-to-random walk  on the symmetric group $\mathbf S_n$ in $L^2$-norm. We show that
the $L^2$-mixing time is of order $n\log n$ (Theorem \ref{th-L2}).

\subsection{Notions of convergence} We will discuss convergence to the uniform distribution using two different distances between probability measures  (or between their densities with respect to the uniform measure $u$). Let $\nu$ be a probability measure on a finite group $G$ and $u$ be the uniform distribution on $G$.

 Total variation (or  $\frac{1}{2}$-$L^1(G,u)$-norm) is defined by
\begin{eqnarray*}\|\nu-u\|_{\mbox{\tiny TV}}&=&\max_{A\subseteq G}\{\nu(A)-u(A)\}\\
&=& \frac{1}{2} \|(\nu/u)-1\|_1= \frac{1}{2}\sum_G|\nu-u|.\end{eqnarray*}

 Convergence in $L^2(G,u)$ is measured using the distance
\begin{eqnarray*} d_2(\nu,u)^2&=& \|(\nu/u)-1\|^2_2 \\
 &=& \sum_{g\in G} |(\nu(g)/u(g))-1|^2 u(g)
= |G|\sum_G|\nu-u|^2.
\end{eqnarray*}

Let $(G_n)_1^\infty$ be a sequence of finite groups such that  $|G_n|$ tends to infinity with $n$. Let $u_n$ be the uniform probability on $G_n$. We say that a sequence of probability measures $\mu_n$ on $G_n$, $n=1,2,\dots$,  has a cut-off at time $t_n$ in $L^p$, $p=1,2$, 
if $t_n\ra \infty$ and, for any $\epsilon>0$, 
$$\lim_{n\ra \infty}d_p(\mu_n^{((1+\epsilon) t_n)},u_n)=0 \mbox{ and  } d_p(\mu_n^{((1-\epsilon)t_n)},u_n) =l_\infty(p)$$
where $l_\infty(1)=1$ and $l_\infty(2)=+\infty$. 

Whenever the probability measure $\mu$ is symmetric, i.e., $\check{\mu}=\mu$, the associated convolution operator $f\mapsto Mf$  is diagonalizable with real eigenvalues $-1\le \beta_{|G|-1}\le \dots \le \beta_1 \le 
\beta_0=1$ and
$$d_2(\mu^{(t)},u)^2= \sum_1^{|G|-1} \beta_i^{2t},\;t=1,2,\dots.$$
Moreover,  $d_\infty(\mu^{(2t)},u)=\max_G|\left\{\frac{\mu}{\nu}-1|\right\}=|G|\mu^{(2t)}(e)-1 =d_2(\mu^{(t)},u)^2$.

Let us illustrate these definitions using the classical examples described above.

\begin{itemize} 
\item (Top-to-random)   Convergence in total variation occurs precisely at time $n\log n$  in the sense that, if we set $t(n,c)=n\log n+ cn$,
$$\lim_{n\ra \infty}\|\mu^{(t(n,c))}-u\|_{\mbox{\tiny TV}}  =\left\{\begin{array}{l} 1 \mbox{ if } c<0\\ 0\mbox{ if } c>0.\end{array}\right.$$
See \cite{Diabook,DFP}. With a little work, the results in \cite{DFP} easily imply \
$$\lim_{n\ra \infty}d_2(\mu^{(t(n,c))},u) =\left\{\begin{array}{lc} \infty &\mbox{ if } c<0\\ 0 &\mbox{ if } c>0.\end{array}\right. $$
\item (Random-to-random)  Convergence in total variation (and in $L^2$) occurs with a cut-off at time $(3n/4)\log n$. 
See \cite{BN}.
\item (Random transposition)  Convergence in total variation (and in $L^2$) occurs with a cutoff at time $(n/2)\log n$.
See \cite{DSh,Diabook,SCZRef}.
\end{itemize}

\section{Single-Card Markov Chain}\label{sec-one}
To investigate the complex behavior of the hit-and-run top-to-random chain, it behooves us to explore the dynamics of just a single card. We do so by defining a Markov chain $(X_t)^\infty_{t = 0}$ with state space $\{1, 2,\dots, n \}$ that represents the position of an arbitrarily chosen card after $t$ 
shuffle's iterations. This is a classical example of a function of a Markov chain that produces a Markov chain.

\subsection{Abstract projection} Before proceeding with the example, we  review some general aspects of this situation.  Abstractly, we start with a Markov kernel $Q$ on a state space $X$ (in our case, 
$Q(x,y)= q_S(x^{-1}y)$ on $\mathbf S_n$) and a lumping (or projection) map $p: X\ra \underline X$  which is surjective and has the property that
$$\sum_{y\in X: p(y)=\underline{y}}Q(x,y) = \underline{Q}(\underline{x},\underline{y})$$
depends only on $p(x)=\underline{x}$. This defines a Markov kernel on $\underline{X}$. If $Q$ has stationary measure $\pi$ then its push-forward $\underline{\pi}(\underline{x})=\pi (p^{-1}(\underline{x}))$ is stationary for $\underline{Q}$. Moreover,
$$\|Q^t(x,\cdot)-\pi\|_{\mbox{\tiny TV}} \ge \|\underline{Q}^t(\underline{x},\cdot)-\underline{\pi}\|_{\mbox{\tiny TV}}.$$
This simple comparison does not work well for the $L^2$ and $L^\infty$ convergence measured using $d_2$ and $d_\infty$ because normalization becomes an issue.  

Let $\beta$ and $\underline{\phi}$ be an eigenvalue and associated eigenfunction for the chain $\underline{Q}$. Then it is plain that the function 
$\phi(x)= \underline{\phi}\circ p(x)$ is an eigenfunction for $Q$ with eigenvalue $\beta$. Also, two orthogonal eigenfunctions $\underline{\phi}_1,\underline{\phi}_2$
for $\underline{Q}$ on $L^2(\underline{\pi})$ give orthogonal $\phi_1,\phi_2$ in $L^2(\pi)$ (we will not use this second fact). 

\subsection{Single card chain in $L^2$}
Let $q$ be the measure for the hit-and-run version of top-to-random defined at (\ref{HRTR}).
We consider the projection of $Q(x,y)=q(x^{-1}y)$ on $\{1,\dots,n\}$ corresponding to following the position of a single card.  To simplify notation, we set
$\underline{Q}=K$ and notice that the stationary (and reversible) measure for $K$ is the uniform measure on $\{1,\dots, n\}$. 
The transition probabilities $K(i,j)$, $i,j\in \{1,\dots, n\}$ are given by
$$K(i,j) = \begin{cases}
\frac{1}{n}\sum\limits_{k \geq j}^n\frac{1}{k} + \frac{i - 1}{n} &\text{if } i = j \\
\frac{1}{n}\sum\limits_{k \geq i}^n\frac{1}{k} &\text{if } j< i \\
\frac{1}{n}\sum\limits_{k \geq j}^n\frac{1}{k} &\text{if } j > i 
\end{cases}
$$ 
\begin{lem}The eigenvalues and associated eigenvectors of the stochastic matrix $(K(i,j))_{1\le i,j\le n}$  are  $\beta_0=1$, $\mathbf{\Psi}_0=(1,\dots,1)$ and
$$\beta_i= 1-\frac{i}{n},\;\mathbf{\Psi}_i= \left(\frac{-1}{n - i},\dots,\frac{-1}{n - i}, 1, 0, \dots, 0\right),\;\; i=1,\dots,n-1,$$
where, in $\mathbf{\Psi}_i$, the value $-1/(n-i)$ is repeated $n-i$ times.  
\end{lem}
\begin{proof} The statement was extrapolated from a direct computation of the case $n=4$. A direct computation then shows that the proposed eigenvectors
are indeed eigenvectors associated with the stated eigenvalues.  These eigenvectors are not normalized and
$$\|\mathbf{\Psi}_i\|_2^2=  \frac{1}{n(n-i)} +\frac{1}{n}= \frac{n-i+1}{n(n-i)}, \;\;i=1,\dots, n-1.$$
\end{proof}
In the next lemma, we use this knowledge to compute
\begin{eqnarray*}d_2(K^t(i,\cdot),u)^2&= &\sum_{j=1}^n \left|K^t(i,j)-\frac{1}{n}\right|^2\\
&=&\sum_{k=1}^{n-1} \beta_k^{2t} \frac{\mathbf{\Psi}_k(i)^2}{\|\mathbf{\Psi}_k\|_2^2}. \end{eqnarray*}
\begin{lem}  The quantity $d_2(K^t(i,\cdot),u)^2$ equals
$$
\begin{cases} 
\displaystyle\sum_{k= 1}^{n - 2} \Big(1 - \frac{k}{n}\Big)^{2t}\frac{n}{(n - k)(n - k + 1)} + \Big(\frac{1}{n}\Big)^{2t}\frac{n}{2} \quad \text{if } i = 1, \\[5ex]
\displaystyle\sum_{k= 1}^{n - i} \Big(1 - \frac{k}{n}\Big)^{2t}\frac{n}{(n - k)(n - k + 1)} + \Big( \frac{i-1}{n}\Big)^{2t}\frac{n(i-1)}{i} \quad \text{if } 1< i < n, \\[5ex]
\Big(1 - \frac{1}{n}\Big)^{2t}(n-1) \quad \text{if } i= n. 
\end{cases}
$$ 
\end{lem}

The term $\sum_{k= 1}^{n - i} \Big(1 - \frac{k}{n}\Big)^{2t}\frac{n}{(n - k)(n - k + 1)}$ can be bounded above by 
$$\frac{1}{n}\sum_{k= 1}^{n - i} \Big(1 - \frac{k}{n}\Big)^{2t-2}$$
and bounded below by  one-half of this quantity.    Set
$$B(n,t,i) \left(1-\frac{1}{n}\right)^{2t-1}=\frac{1}{n}\sum_{k= 1}^{n - i} \Big(1 - \frac{k}{n}\Big)^{2t-2}$$\begin{lem}   For $n\ge 4$, $t\ge 1$, the quantity $B(n,t)$ is bounded as follows:
\begin{itemize}
\item  If  $2\le i\le an$,  $a \le1/2$,  
$$\left(\frac{1}{n-1}+\frac{1}{4(2t-1)}\right)\le B(n,t,i)\le \left(\frac{1}{n-1}+\frac{1}{2t-1}\right).$$
 \item If $i\le an$,  $ a <1$, and  $n\ge 2/(1-a)$, then there exists $c_a>0$ such that 
 $$\left(\frac{1}{n-1}+\frac{c_a}{2t-1}\right)\le B(n,t,i)\le \left(\frac{1}{n-1}+\frac{1}{2t-1}\right).$$ 
 \item If $n-i_0\le i\le n-2$,  
 $$\frac{1}{n-1} \le B(n,t,i)\le \frac{i_0}{n-1}.$$ \end{itemize}
\end{lem}

These elementary estimates give the following result.
\begin{pro}  (a) For each fixed $i=1,2,3,\dots $, set $t_i(n,c)=\frac{n}{2i}(\log n \,+c)$.  Then
$$\lim_{n\ra \infty} d_2(K^{t_i(n,c)}(n-i,\cdot),u) =\left\{\begin{array}{cl} +\infty &\mbox{ if } c<0,\\
0 &\mbox{ if } c>0.\end{array}\right.$$
That is, the position of the card starting in position $n-i$ becomes random in $L^2$-sense with a cut-off at time $\frac{i}{2n}\log n$. 

(b) For each fixed $i=1,2,3,\dots$ and any $t_n$ tending to infinity,
$$\lim_{n\ra \infty} d_2(K^{t_n}(i,\cdot),u) = 0.$$
Moreover, there exists a constant $c_i>0$ such that for any  $\epsilon \in (2/n,1)$, 
$$d_2(K^{t}(i,\cdot),u)=\epsilon   \Rightarrow t(n) \in (c_i/\epsilon^2, 10/\epsilon^2).$$

(c)  For each fixed $a\in (0,1)$  set $t_a(n,c)= \frac{1}{2\log (1/a)} (\log n +c)$.  Then
$$\lim_{n\ra \infty} d_2(K^{t_a(n,c)}([an],\cdot),u) =\left\{\begin{array}{cl} +\infty &\mbox{ if } c<0,\\
0 &\mbox{ if } c>0.\end{array}\right.$$
That is, the position of the card starting in position $[an]$ becomes random in $L^2$-sense with a cut-off at time $\frac{\log n}{2\log (1/a)}$. 
\end{pro}

We end this section with two eigenvalue data plots which concern different projections, namely, those corresponding to following a given pair or a triplet of cards instead of just one. These chains become more complex and we have not computed all there eigenvalues and eigenfunctions. Instead, this plots are based on computer assisted computations of the eigenvalues of these chains.
The first plot  is for the two-card chain on 21 cards and the second is for the  three-card chains on 21 cards.

\begin{figure}[h!]
  \includegraphics[width=0.5\linewidth]{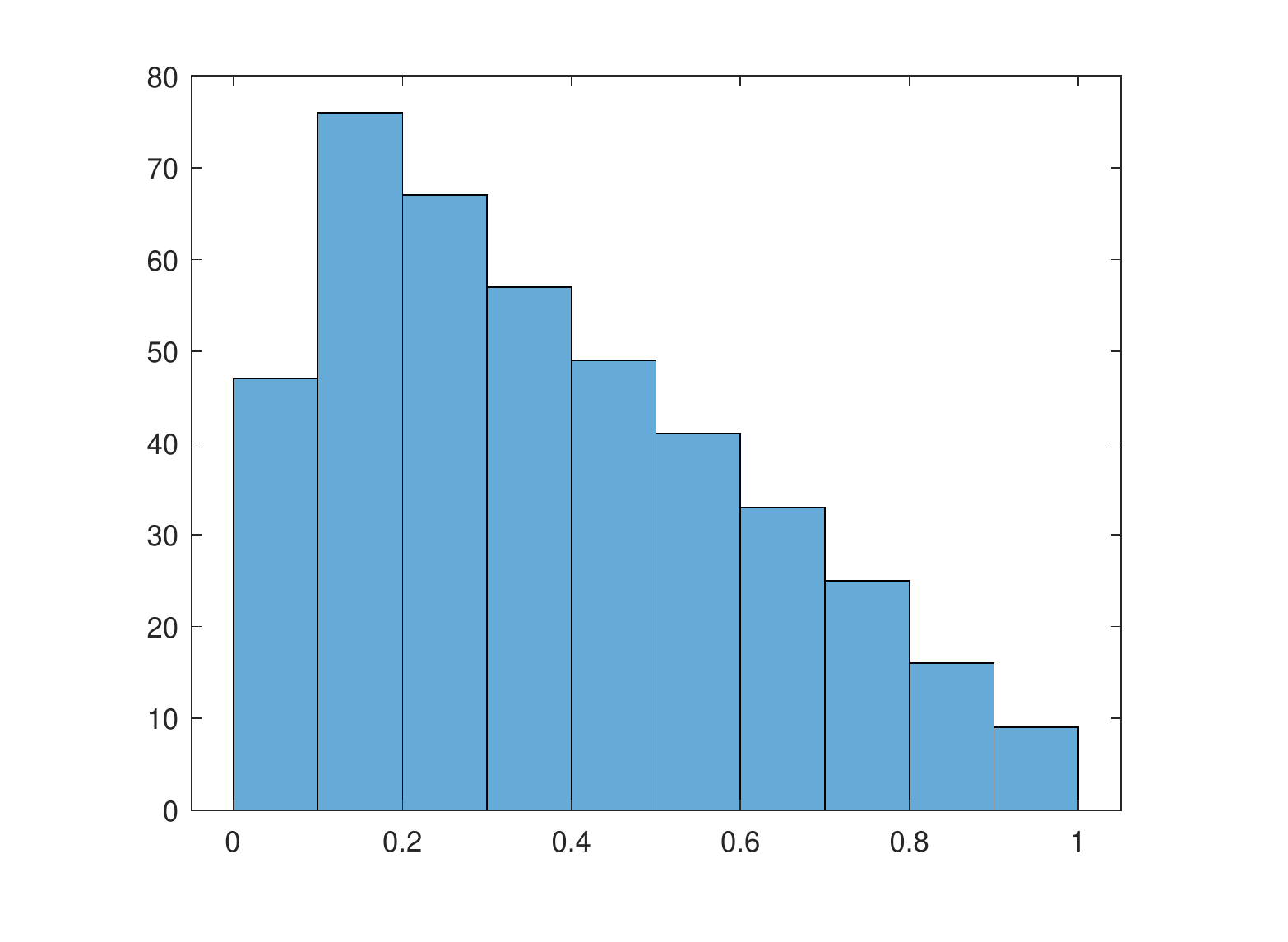}\;\;\; \includegraphics[width=0.55\linewidth]{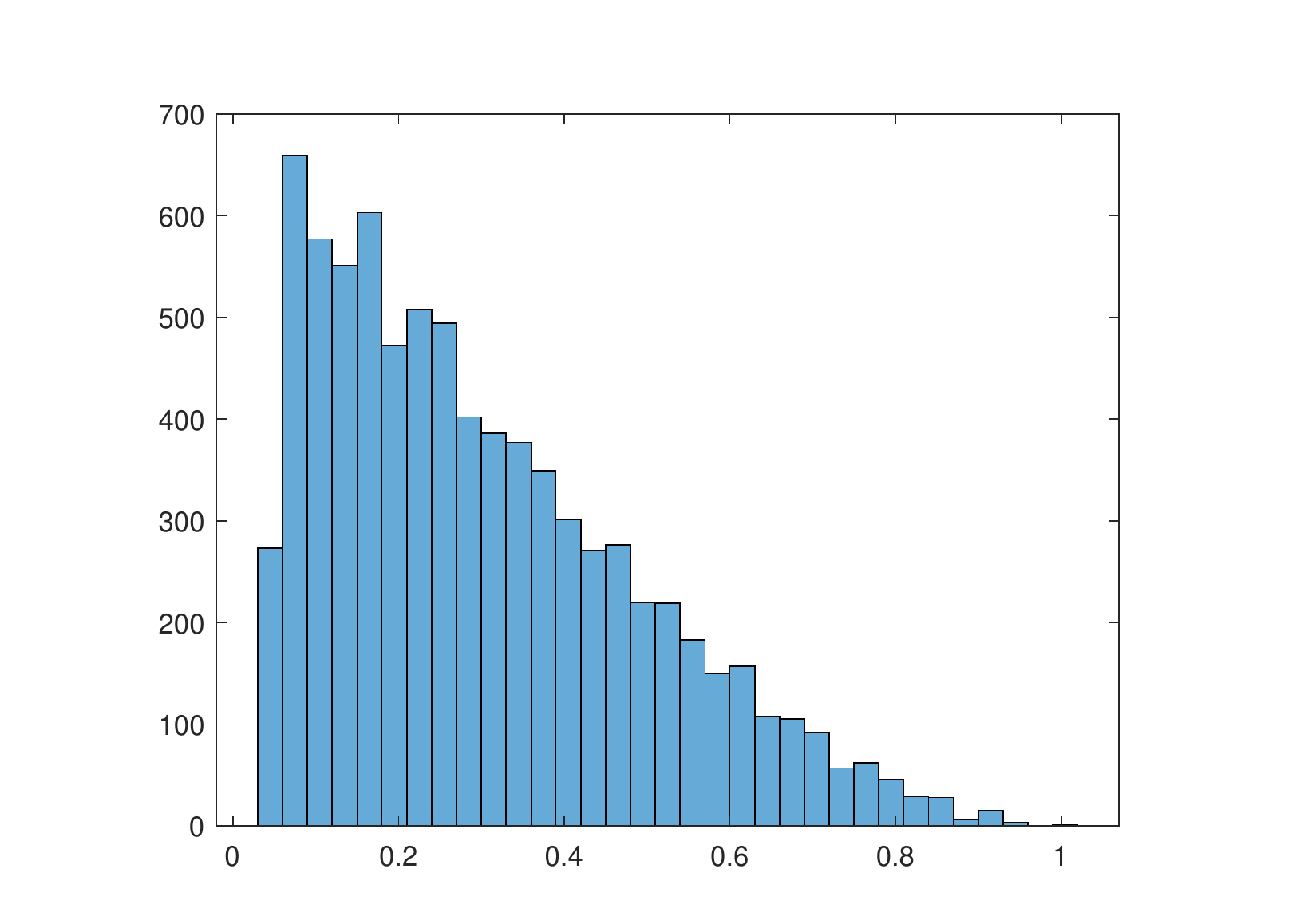}  \caption{21 cards. Left: The eigenvalue distribution for the two-card chains; Right: The eigenvalue distribution for the three-card chains. Note that all eigenvalues are positive.}
  \label{fig:two21ref}
\end{figure}

\subsection{Single card chain in $L^1$}

The relatively simple form of the eigenvalues and eigenvectors of the single card chain $K$ also allows us to determine the $L^1$-distance of $K^{t}(i,\cdot)$ from its stationary measure $u$. Namely, the diagonalization of $K$ shows that the $i$-th row of $K^t$, $K^t(i,\cdot)$, $1\le i\le n-1$, consists of the repeated entry
$$\frac{1}{n^t}\left(\frac{-(i-1)^t}{i} + \frac{i^{t-1}}{i+1} + \dots + \frac{(n-1)^{t-1}}{n}\right) + \frac{1}{n}$$
in columns $j=1$ through $i-1$,
$$\frac{1}{n^t}\left(\frac{(i-1)^{t+1}}{i} + \frac{i^{t-1}}{i+1} + \dots + \frac{(n-1)^{t-1}}{n}\right) + \frac{1}{n}$$ in column $i$,
$$\frac{1}{n^t}\left(\frac{-(i+k-1)^t}{i+k} + \frac{(i+k)^{t-1}}{i+k+1} + \dots + \frac{(n-1)^{t-1}}{n}\right) + \frac{1}{n}$$
in column $j=i+k$, $i+1\le i+k<n$, and 
$$\frac{1}{n^t}\left(\frac{-(n-1)^t}{n}\right) + \frac{1}{n}$$ in column $n$.  The last row, $i=n$, consists of the entries
$$\frac{1}{n^t}\left(\frac{-(n-1)^t}{n}\right)+\frac{1}{n} $$ in columns $1$ through $n-1$ and  
$$\frac{1}{n^t}\left(\frac{(n-1)^{t+1}}{n} \right)+ \frac{1}{n} $$ 
in column $n$.

In the case $i=n$ (single card starting at the bottom of the deck), we find that
$$\|K^t(n,\cdot)-u\|_{\mbox{\tiny TV}}= \left(1-\frac{1}{n}\right)^{t+1}$$
(indeed, this card position is uniform as soon as it is touched). 
For $1\le i\le n-1$,
\begin{eqnarray*} 2\|K^t(i,\cdot)-u\|_{\mbox{\tiny TV}}&=& \frac{1}{n^t}\left|-\frac{(n-1)^{t}}{n}\right|+\frac{i-1}{n^t}\left|\frac{-(i-1)^t}{i} + \sum_{\ell=i}^{n-1}\frac{\ell^{t-1}}{\ell+1}\right| \\ &&+  
\frac{1}{n^t}\left( \frac{(i-1)^{t+1}}{i}+
\sum_{\ell=i}^{n-1}\frac{\ell^{t-1}}{\ell+1}\right)\\
&&+ \frac{1}{n^t} \sum_{k=1}^{n-i-1} \left|\frac{-(i+k-1)^t}{i+k} + \sum_{\ell=1}^{n-(i+k)}\frac{(i+k+\ell-1)^{t-1}}{i+k+\ell} \right| \\
&=&J_1+J_2+J_3+J_4.\end{eqnarray*}
Looking at $J_2$ and $J_4$ for large $t$, i.e., $t\ge (n-1)\log n +\frac{n-1}{n-2}$, we have
$$\frac{-(i-1)^t}{i} + \sum_{\ell=i}^{n-1}\frac{\ell^{t-1}}{\ell+1}\ge 0$$
and 
$$\frac{-(i+k-1)^t}{i+k} + \sum_{\ell=1}^{n-(i+k)}\frac{(i+k+\ell-1)^{t-1}}{i+k+\ell}\ge 0,\;\;\mbox{  for } k\in\{ 1,\dots, n-i-1\}.$$
Because the sum of all the same terms in $J_1+J_2+J_3+J_4$ but without any absolute value is $$\sum_{\ell=1}^n K^t(i,\ell)-\pi(\ell)=0,$$ it follows that, for  $t\ge (n-1)\log n +\frac{n-1}{n-2}$, $2\|K^t(i,\cdot)-u\|_{\mbox{\tiny TV}}=2|J_1|$, that is,
$$\|K^t(i,\cdot)-u\|_{\mbox{\tiny TV}}=\frac{1}{n}\left(1-\frac{1}{n}\right)^{t},\;\;i\in \{1,\dots, n-1\}.$$
This, of course, occurs only much after approximate convergence has taken place. It only describe the long term asymptotic behavior
of $\|K^t(i,\cdot)-u\|_{\mbox{\tiny TV}}$, $i<n$. To describe the shorter term behavior, we consider three cases:  Botton starting positions  of the type $n-i$ for fixed $i=1,2,\dots$, top starting positions of the type $i=1,2, \dots,$ and middle of the pack starting positions of the type $[an]$, $a\in (0,1)$.  

For starting position $n-i$, $i$ fixed, we get a reasonable lower bound by focussing on the first and third terms. Write
\begin{eqnarray*}2 \|K^t(n-i,\cdot)-u\|_{\mbox{\tiny TV}}&\ge & \frac{1}{n}\left(1-\frac{1}{n}\right)^t\\
&&+\frac{1}{n^t}\left( \frac{(n-i-1)^{t+1}}{n-i}+\sum_{\ell=n-i+1}^{n}\frac{(\ell-1)^{t-1}}{\ell}\right)\\
&\ge &\frac{1}{n}\left(1-\frac{1}{n}\right)^t+ \left(1-\frac{i+1}{n}\right)^{t+1}\end{eqnarray*}
An upper-bound of the type
$$  \|K^t(n-i,\cdot)-u\|_{\mbox{\tiny TV}}\le C_i\left(\frac{1}{n}\left(1-\frac{1}{n}\right)^t+ \left(1-\frac{i+1}{n}\right)^{t}\right)$$
holds as well. This proves convergence in time of order $n/(i+1)$ with no cut-off for the bottom cards.

For starting position $i$, $i$ fixed (starting position towards the top), we have
$$2 \|K^t(i,\cdot)-u\|_{\mbox{\tiny TV}}\ge  \left(\frac{1}{n} +\frac{c_i}{t}\right)\left(1-\frac{1}{n}\right)^t .$$
The term $\frac{c_i}{t}(1-1/n)^t$ is contributed by the third and last summand in the general formula. In this last summand, namely,
$$  \frac{1}{n^t} \sum_{k=1}^{n-i-1} \left|\frac{-(i+k-1)^t}{i+k} + \sum_{\ell=1}^{n-(i+k)}\frac{(i+k+\ell-1)^{t-1}}{i+k+\ell} \right|  ,$$
restrict the first summation to those $k$ less than, say, $n/4$. In this range of $k$ values, the positive term in the absolute value dominates the negative term and we obtain a lower bound of the type (we assume $t\ge 4$)
\begin{eqnarray*} \frac{c_i}{n^t} \sum_{k=1}^{n/4} \sum_{ \ell=n/2}^{n-1}\ell ^{t-2}  &\ge &c'_i  \left(1-\frac{1}{n}\right)^{t-1}
\left(\frac{1}{n-1}\sum_{n/2}^{n-1} \left(\frac{\ell}{n-1}\right)^{t-2}\right)\\
&\ge &\frac{c''_i }{t} \left(1-\frac{1}{n}\right)^t    \end{eqnarray*}
where  we used an integral to lower bound the Riemann sum in parentheses. A matching upper-bound,
$$ \|K^t(i,\cdot)-u\|_{\mbox{\tiny TV}}\le  C_i \left(\frac{1}{n} +\frac{1}{t}\right)\left(1-\frac{1}{n}\right)^t $$
is easily obtained. The key rate of convergence is thus in $1/t$ for the top starting positions.

For starting position in the middle of the pack, $i=[an]$, $a\in (0,1) $ fixed, a similar analysis shows that
$\|K^t([an],\cdot)-u\|_{\mbox{\tiny TV}}$ is also of order 
$$\left(\frac{1}{n}+ \frac{1}{t}  \right)\left(1-\frac{1}{n}\right)^t .$$
This time, we use the second term,
$$ \frac{i-1}{n^t}\left|\frac{-(i-1)^t}{i} + \sum_{\ell=i}^{n-1}\frac{\ell^{t-1}}{\ell+1}\right| $$
to obtain a lower bound of the type $c_a(1-1/n)^t/t$.  Indeed, for $i=[an]$ and $n$ large enough,
\begin{eqnarray*} \sum_{\ell=i}^{n-1}\frac{\ell^{t-1}}{\ell+1}&\ge &\frac{(n-1)^{t-1}}{2} \sum_{\ell=[an]}^{n-1} \left(\frac{\ell}{n-1} \right)^{t-2}\frac{1}{n-1}\\
&\ge & \frac{(n-1)^{t-1}}{2} \int _{(a+1)/2}^1 x^{t-2}dx\ge   c_a \frac{(n-1)^{t-1}}{t-1}.
\end{eqnarray*}
For  $t\ge t_a$, this is larger than  twice   $(i-1)^{t-1}$, $i=[an]$.  It follows that
  \begin{eqnarray*}
  \frac{i-1}{n^t}\left|\frac{-(i-1)^t}{i} + \sum_{\ell=i}^{n-1}\frac{\ell^{t-1}}{\ell+1}\right| &\ge & \frac{c_a}{2} \frac{an}{n^{t} }\frac{(n-1)^{t-1}}{t-1}\\
  &\ge & \frac{c'_a}{t} \left(1-\frac{1}{n}\right)^t .\end{eqnarray*}

\section{Hit-and-run for top-to-random in $L^2$}\label{sec-HRTR2}
 \begin{figure}[h!]
\hspace{-.2in}  \includegraphics[width=0.3\linewidth,trim= 1in 3.7in 0.8in 3in]{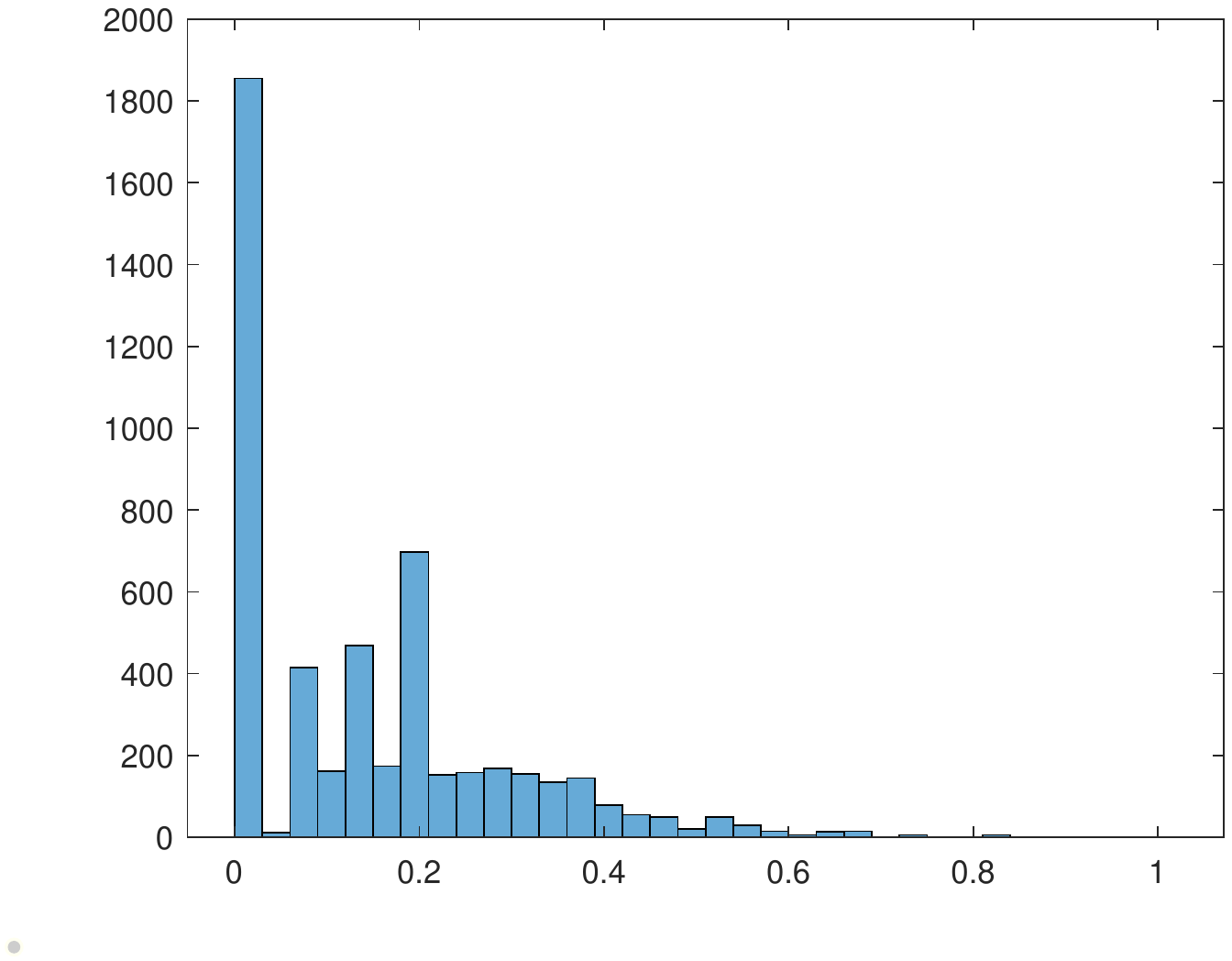}\hspace{-.3in}
  \includegraphics[width=0.3\linewidth,trim= 1in 3.7in 0.8in 3in]{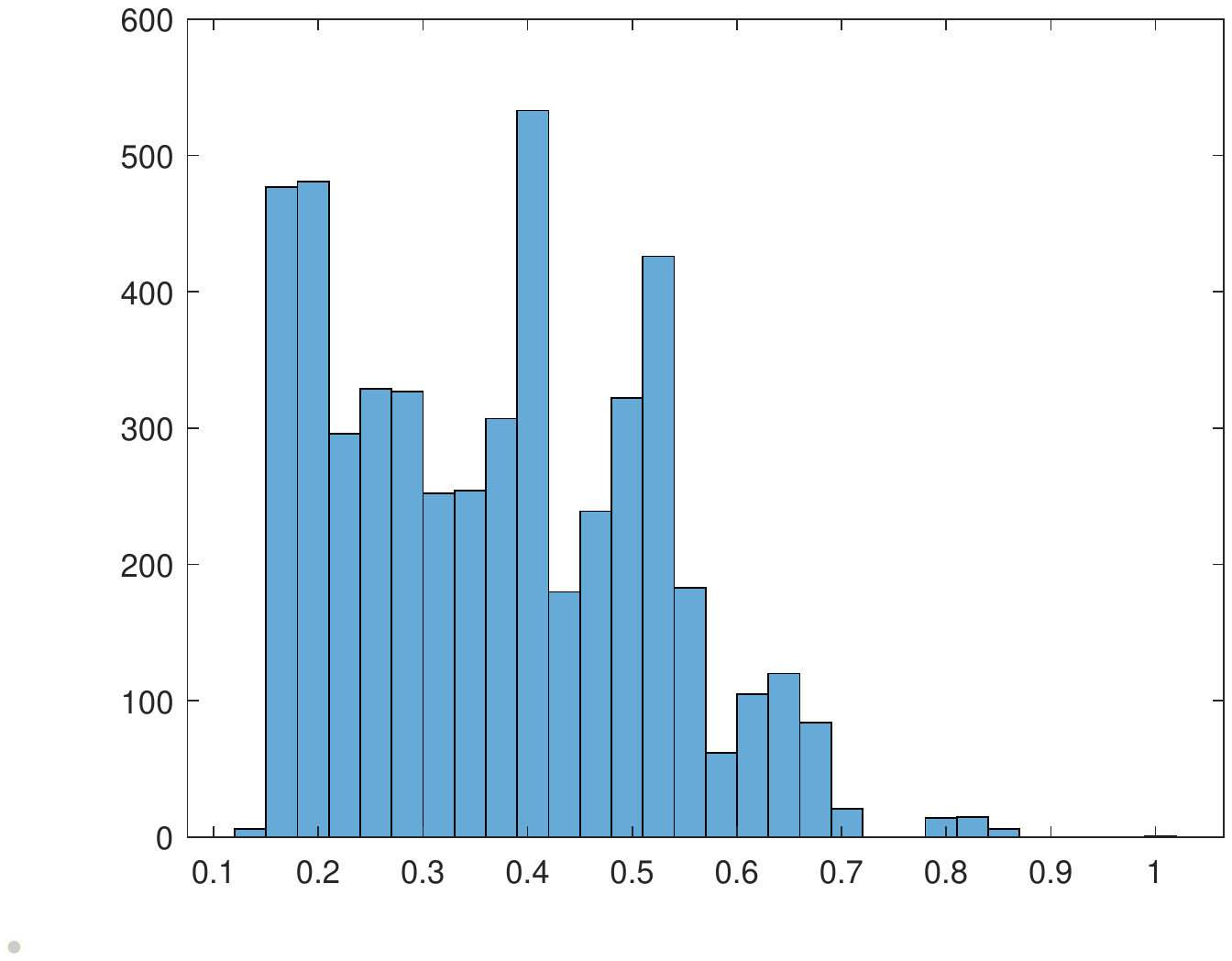}  \includegraphics[width=0.5\linewidth,trim= 0.in 1in 0.8in 3in]{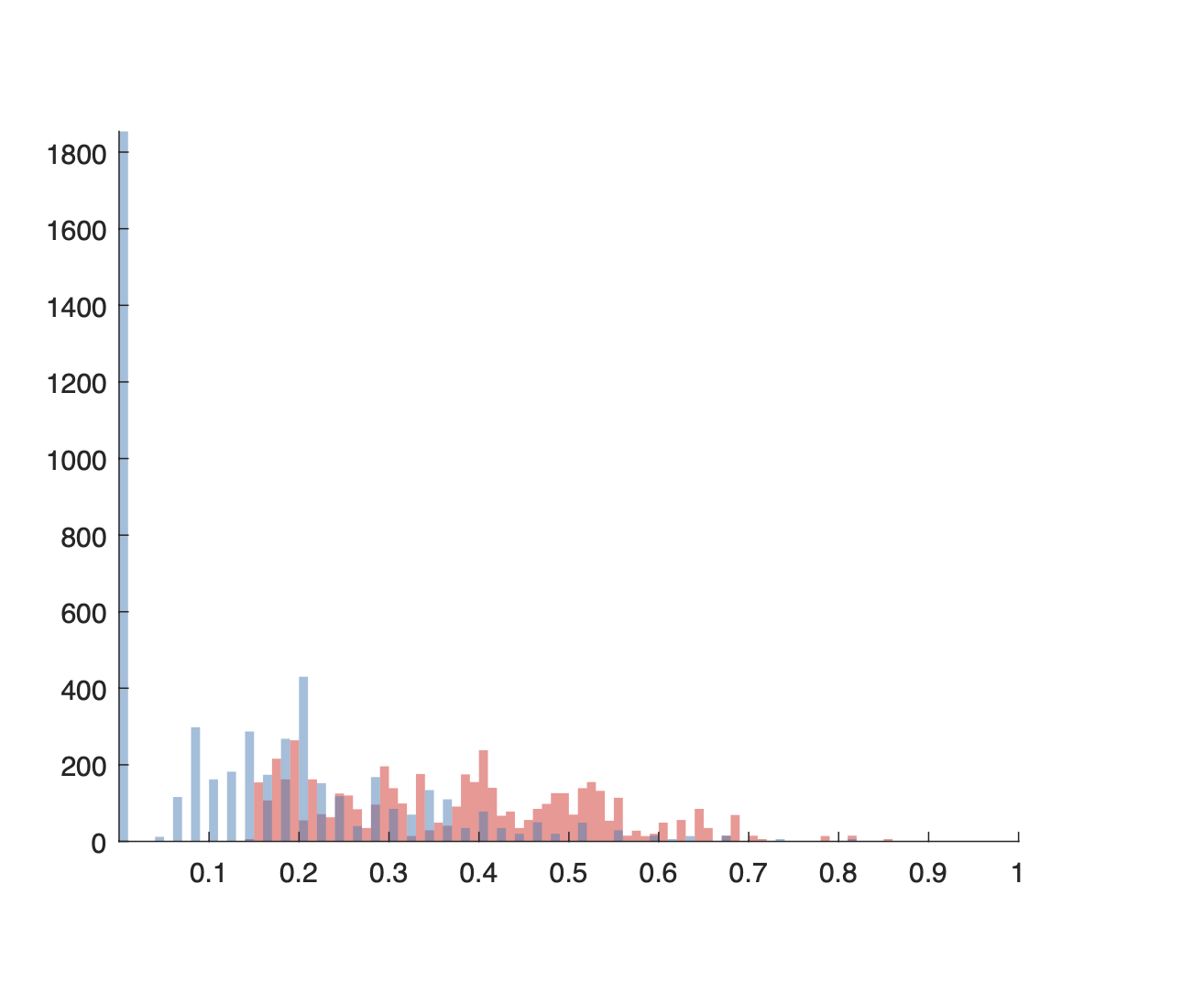}  
   \caption{Comparison of the spectrum of random-to-random (left most graphic, blue on the right most graphic) and hit-and-run for top-to-random (middle graphic, red in the right most graphic). The key difference is the higher multiplicity of very small eigenvalues in the random-to-random shuffle (most of those are actually equal to $0$). Note the different scales on the $y$ axes of the two left most graphics.}
  \label{fig:comp}
\end{figure}

In this section, we present the best results we know regarding the hit-and-run version of top-to-random driven by the measure $q$ at (\ref{HRTR})
when convergence to stationarity is measured in $L^2$-sense, that is, using $d_2(q^{(t)},u)$.
\begin{theo} \label{th-L2} 
For any $n,t$, we have
$$d_2(q^{(t)},u) \ge \sqrt{n-1} \,\left(1-\frac{1}{n}\right)^{t}.$$
The second largest eigenvalue $\beta_1$ of $q$ is bounded by $\beta_1\le 1-1/(8n)$ and, for any $n$ large enough and 
$t(n,c) \ge 9 n\log n \, + 12 n c$, $c>0$,
$$d_2(q^{(t(n,c))}
,u)\le \sqrt{5}e^{-c}.$$
\end{theo}    
\begin{rem} The spectral gap $\lambda=1-\beta_1$ for $q$ is at least  $1/n$ and the time to stationarity in $L^2$, $T$, is at least $\frac{1}{2} n\log n$ so that the product $\lambda T$ tends to infinity. It thus follows from \cite{CSC} that there is a cut-off for this walk on $\mathbf S_n$. The cut-off time is of order $n\log n$ but we  do not know its exact behavior.
\end{rem}
\begin{proof}[Proof of the lower-bound]  In the section concerning following a single card, we learned that $(1-1/n)$ is an eigenvalue of that chain and, consequently, also an eigenvalue of  convolution by $q$ on $\mathbf S_n$. Now, on $\mathbf S_n$, each eigenvalue has multiplicity at least equal to the dimension of any irreducible representation at which it occurs. The group $\mathbf S_n$ has two representations of dimension $1$, the trivial representation and the sign representation. All other irreducible representations have dimension at least $n-1$. So, it suffices to verify that $(1-1/n)$ does not occur only at the sign representation. This can be seen from the form of the associated eigenvector we have constructed. Alternatively, one easily computes the eigenvalue for the sign representation to be $1/2$ if $n$ is even and $(n-1)/2n$ if $n$ is odd. In any case, this gives the lower-bound
$d_2(q^{(t)},u)\ge \sqrt{n-1}(1-1/n)^t$  as stated.\end{proof}

To prove the stated upper-bound for $d_2(q^{(t)},u)$, we use the comparison technique from \cite{DSCcompG}.  Anticipating on the next section, we use the fact that hit-and-run walks have non-negative spectrum. It turn out that the most efficient comparison is with the random-to-random walk of \ref{basicexa} driven by the measure
$$\mu(\sigma)=\left\{\begin{array}{cl} 1/n & \mbox{ if }  \sigma=\mbox{id}\\
2/n^2 &\mbox{ if } \sigma=\sigma_{i(i+1)}=\sigma_{(i+1)i},\\
1/n^2 &\mbox{ if } \sigma=\sigma_{ij}, \; 1\le i\neq j\le n,  |j-i|>1\end{array}\right.$$
where 
$\sigma_{ij}=(j,j-1,\dots,i)$, $1\le i<j \le n$ . 
Recall that the Dirichlet form associated with a symmetric probability measure $\nu$ on a finite group $G$ is
$$\mathcal E_\nu(u,v)= \frac{1}{2|G|}\sum_{x,y\in G}(u(xy)-u(x))(v(xy)-v(x))\nu(y).$$
\begin{lem} The Dirichlet form $\mathcal E_\mu$  associated with the random-to-random measure $\mu$ and the Dirichlet form $\mathcal E_q$  associated with hit-and-run version of top-to-random satisfy
$$\forall \, u\in L^2(G),\;\;\mathcal E_\mu (u,u)\le 8 \mathcal E_q(u,u).$$
\end{lem}
\begin{proof}  For each $\sigma_{ij}, 1\le i\neq j\le n$, we find a product of $\sigma_k^\ell$, $1\le \ell < k\le n$, which equals $\sigma_{ij}$. There are many way to do this but the following is efficient. Observe that for $i < j$, $\sigma_{ij} = \sigma_j^i\sigma_{j-1}^{j-i}$. That is, to move the card in position $i$ down to position $j$, insert the first $i$ cards at position $j$, then insert the first $j - i$ cards now on top at position $j - 1$. After the first move, the card originally in position $j$ is at position $j - i$, so the second move places it in position $j - 1$. The other $i - 1$ cards moved down to position $j - 1$ are returned to their original spot in the second move (barring the card originally in position $i$) by sliding past them all the cards they originally slid past, which were on the top after the first move. For $i > j$, $\sigma_{ij} = \sigma_{ji}^{-1} = \sigma_{i-1}^{-(i-j)}\sigma_i^{-j} = \sigma_{i-1}^{j-1}\sigma_i^{i-j}$. 
Now we use \cite[Theorem 1]{DSCcompG} with $\widetilde{\mathcal E}=\mathcal E_mu$, $\mathcal E=\mathcal E_q$ which gives
$$\mathcal E_\mu\le A\mathcal E_q,\;\;A=\max_{\sigma: q(\sigma)>0}\left\{\frac{1}{q(\sigma)} \sum_{1\le i\neq j\le n} |\sigma_{ij}|N(\sigma,\sigma_{ij})\mu(\sigma_{ij})\right\}.$$
In the formula giving $A$, $|\sigma_{ij}|$ is the length of the product expressing $\sigma_{ij}$, which, in our case, is always equal to $2$; $N(\sigma,\sigma_{ij})$ is the number of time $\sigma$ appears in the product for $\sigma_{ij}$. So, if $\sigma=\sigma_k^{\ell}$
for some $2\le k\le n$ and $1\le \ell\le k-1$, $1\le i<j\le n$,
$$N(\sigma,\sigma_{ij})=\left\{\begin{array}{cl} 0 & \mbox{ if }  (k,\ell)\not\in\{(j,i),(j-1,j-i)\} \\
 1 &\mbox{ if }  (k,\ell)\in\{(j,i),(j-1,j-i)\}. \end{array}\right.$$
 When $1\le j<i\le n$, we similarly have 
 $$N(\sigma,\sigma_{ij})=\left\{\begin{array}{cl} 0 & \mbox{ if }  (k,\ell)\not\in\{(i-1,j-1),(i,i-j)\} \\
 1 &\mbox{ if }  (k,\ell)\in\{(i-1,j-1),(i,i-j)\}. \end{array}\right.$$

 For $1\le \ell<k<n$, this gives
 $$
 \left\{\frac{1}{q(\sigma_k^\ell)} \sum_{1\le i\neq j\le n} |\sigma_{ij}|N(\sigma,\sigma_{ij})\mu(\sigma_{ij})\right\}=
 8k/n $$
 whereas  for $(k,\ell)=(n,\ell)$ the result is $4$.  
  \end{proof}
  \begin{proof}[Proof of the upper-bound in Theorem \ref{th-L2}]  Given the comparison inequality
  $\mathcal E_\mu\le 8 \mathcal E_q$ between quadratic forms, Lemma 6 of \cite{DSCcompG} (see also \cite[Theorem 10.2]{SCRW})
  provides a comparison inequality. Here we used the same idea in a slightly tighter way.  Let $0\le \alpha_{|G|-1}\le \dots\le \alpha_1<\alpha_0=1$
  be the eigenvalues for random-to-random. The inequality   $\mathcal E_\mu\le 8 \mathcal E_q$ gives
  $1-\beta_i\ge \frac{1}{8} (1-\alpha_i)$. Split the sum
  $$  d_2(q^{(t)},u)^2 =\sum_{i=1}^{|G|-1} \beta_i^{2t} \le \sum_1^{|G|-1}e^{-2t(1-\beta_i)}$$ 
  into two sums, the  sum over those indices $i$ such that $\alpha_i\le 1/2$ and the sum over the others indices.  For the first sum, write
 $$\sum_{i: \alpha_i\le 1/2}e^{-2t(1-\beta_i)}\le (n!) e^{- t/8}.$$
 For the second sum, note that $e^{-3(1-x)/2}\le x$ when $x\in[.5,1]$, and write
\begin{eqnarray*}
\sum_{i: \alpha_i\ge 1/2} e^{-2t(1-\beta_i)}&\le &\sum_{i: \alpha_i\ge 1/2} e^{-2t(1-\beta_i)}\\
&\le &\sum_{i: \alpha_i\ge 1/2} e^{-t(1-\alpha_i)/4}\\ 
&\le &\sum_{i: \alpha_i\ge 1/2} \alpha_i ^{ t/6} \le d_2(q^{([t/12])},u)^2 .\end{eqnarray*}
This gives
\begin{equation}\label{d-comp}
 d_2(q^{(t)},u)^2\le (n!)e^{-t/8} +d_2(q^{([t/12])},u)^2.\end{equation}
 In \cite{BN}, it is proved that the spectral gap for $\mu$ is asymptotically equal to $1/n$ and that   
 $$d_2(\mu^{(s)},u)^2\le 4 e^{-2c} \mbox{ for any } s\ge \frac{3}{4} n \log n +cn,\;c>0,$$
  as long as $n$ is sufficiently large (let us note that  this is a rather difficult result).   Using this in (\ref{d-comp}) yields the upper-bound stated in Theorem \ref{th-L2}.
  \end{proof}

   \section{Positivity of the spectrum}  \label{sec-pos}
   In this final section, we prove Theorem \ref{th-pos}.  Given a general hit-and-run random walk 
   driven the measure $q_S$ at (\ref{def-HR}) on a finite group $G$, we set
   $$Q(x,y)= q_S(x^{-1}y)=\frac{1}{k} \sum_{i=1}^k \frac{1}{m_i} \sum_{j=0}^{m_i-1} \delta_{s_i^j}(x^{-1} y), \quad x,y \in G.$$
   This defines  a self-adjoint operator $f\mapsto Qf=\sum _yQ(\cdot,y)f(y)$ acting on the space $H=L^2(G)$ equipped with the inner product
   $$\langle f_1,f_2 \rangle = \frac{1}{\vert G \vert} \sum_{x\in G} f_1(x) f_2(x).$$  Let 
   $$\beta_{|G|-1}\le\dots \le \beta_1\le \beta_0=1$$ 
be the $|G|$ eigenvalues of  this operator. Because $Q$ is Markov, these eigenvalues are contained in the interval $[-1,1]$.
The theorem we want to prove, Theorem \ref{th-pos}, asserts that they are in fact in the interval $[0,1]$, that is, that $Q$ is non-negative in the sense that
$\langle Qf,f\rangle \ge 0$. The proof follows the main idea of \cite{RUpos} which consists in writing $Q$ in the product  form $$Q= P^*RP$$ using  auxiliary operators $P,R,P^*$ where $R=R^2$ is acting on the  extended Hilbert space 
$H_{\text{aux}}$, the space of functions on $G\times \{1,\dots,k\}$ equipped with its natural inner product $\langle \cdot,\cdot\rangle _{\text{aux}}$. Because $R=R^2$ and $P^*$ is the formal adjoint of $P$, such a decomposition establishes that $$\langle Qf,f\rangle=\langle P^*RP f,f\rangle =\langle  RPf,RPf\rangle_{\text{aux}}\ge 0.$$ To use such a decomposition is a key insight from \cite{RUpos}  but it also appears in \cite[Remark~4.4]{Ul14} and \cite[Lemma~3.1]{dyer2014structure}.

Define the auxiliary Hilbert space $H_{\text{aux}} = \mathbb{R}^{G\times \{1,\dots, k\}}$ equipped with inner-product
\[
	\langle g_1,g_2\rangle_{\text{aux}} 
:= \frac{1}{k \vert G \vert} \sum_{x\in G} \sum_{i=1}^k g_1(x,i) g_2(x,i),
\]
where $g_1,g_2 \in H_{\text{aux}}$. Let $P\colon H \to H_{\text{aux}}$ denote the bounded linear operator given by
\[
	P f(x,i) = f(x), \quad (x,i) \in G\times\{1,\dots,k\}.
\] 
Note that the adjoint operator, $P^*\colon H_{\text{aux}} \to H$ of $P$, is given  by
$$P^* g(x) = \frac{1}{k}\sum_{i=1}^k g(x,i).$$
That $P^*$ is the adjoint of $P$ means here that 
$\langle P^*g,f \rangle = \langle g,Pf \rangle_{\rm aux}$ for any $f\in H$ and $g\in H_{\text{aux}}$, which can be verified by a simple calculation. 
The kernels of these operators are
\begin{align}
P((x,i),y) & = \delta_x(y),\\
P^*(x,(y,i)) & = \frac{\delta_{x}(y)}{k}, 
\end{align} 
for any $x,y \in G$ and $i\in \{1,\dots,k\}$.  For any pair $(x,i)\in G\times\{1,\dots,k\}$, call
\[
	\mathcal{Z}(x,i) := \{ x_0,\dots,x_{m_i-1} \mid x_i := x s_i^j, j=0,\dots, m_i-1 \}
\]
the orbit of $x$ in $G$ under the cyclic subgroup $\langle s_i\rangle=\{ s_i^j: j=0,\dots, m_i-1\}$ generated by $s_i$ in $G$.
Define a  Markov transition kernel $$R(\cdot,\cdot) \mbox{ on  }(G\times \{1,\dots,k\})^2$$  by setting
\[
	R ((x,i_1), (y,i_2)) := \delta_{i_1}(i_2) \frac{\delta_{\mathcal{Z}(x,i_1)}(y)}{m_{i_1}}.
\]
It induces a  Markov operator, $R \colon H_{\text{aux}} \to H_{\text{aux}}$,  given by
\[
	Rg (x,i) = \frac{1}{m_i} \sum_{z\in \mathcal{Z}(x,i)} g(z,i), \qquad g\in H_{\rm aux}.
\]
Because  
$$x\in \mathcal{Z}(y,i) \;\mbox{ if and only if }\; y\in \mathcal{Z}(x,i)$$ the operator $R$ is symmetric, i.e.,
\begin{eqnarray*}
	R((x,i_1),(y,i_2)) 
	&=&  \delta_{i_1}(i_2) \frac{\delta_{\mathcal{Z}(x,i_1)}(y)}{m_{i_1}} 
	\\
	&=& \delta_{i_2}(i_1) \frac{\delta_{\mathcal{Z}(y,i_2)}(x)}{m_{i_2}} 
	= 	R((y,i_2),(x,i_1)). 
\end{eqnarray*}
Thus, the corresponding operator $R \colon H_{\rm aux} \to H_{\rm aux}$ is self-adjoint.
\begin{lem} The operator $R$ satisfies $R^2=R$ and
$Q=R^* RP$.
\end{lem}
\begin{proof}[Proof of $R^2=R$] For arbitrary $g\in H_{\rm aux}$ we have
\begin{align*}
	R^2 g(x,i) 
	& = \frac{1}{m_i} \sum_{y\in \mathcal{Z}(x,i)} Rg(y,i)
	  = \frac{1}{m_i} \sum_{y\in \mathcal{Z}(x,i)} \sum_{z\in \mathcal{Z}(y,i)} \frac{g(z,i)}{m_i}
	  = Rg(x,i).
\end{align*}
Here we used the facts  that for $y\in \mathcal{Z}(x,i)$ we have
$
 \mathcal{Z}(x,i) = \mathcal{Z}(y,i)
$ 
and $\vert \mathcal{Z}(x,i)\vert = m_i$. \end{proof}

\begin{proof}[Proof of $Q=P^*RP$]
For any $x,y \in G$ we have
\begin{align*}
	P^* R P (x,y) 
& = \frac{1}{k} \sum_{i=1}^k R P ((x,i),y) 
  = \frac{1}{k} \sum_{i=1}^k \sum_{z\in \mathcal{Z}(x,i)} \frac{P((z,i),y)}{m_{i}} \\
& = \frac{1}{k} \sum_{i=1}^k 
		\frac{1}{m_i}\sum_{z\in \mathcal{Z}(x,i)} \delta_{y}(z)
  = \frac{1}{k} \sum_{i=1}^k 
  \frac{1}{m_i} \sum_{j=0}^{m_i-1} \delta_{y}(xs_i^j)\\
& = Q(x,y). 
\end{align*}
Here we used the definition of $\mathcal{Z}(x,i)$ and, in the last equality, that $
	y=xs_i^j $ if and only if $s_i^j = x^{-1} y.$
\end{proof}

\section{Final remarks}

\begin{exa}[Example where hit-and-run is faster] Assume that $G=(\mathbb Z/n\mathbb Z)^d$ and $S=(0,e_1,-e_1,\dots,e_d,-e_d)$ where $e_j=(0,\dots, 0,1,0, \dots,0)$ with the $1$ in position $j$, $1\le j\le d$. 
In $L^2$ and $L^1$, the walk driven by $\mu_S$ mixes in time $ \frac{d\log d}{2(1-\cos 2\pi/n)}$ (as $d$ (and possibly) $n$ tends to infinity).  The measure $q_S$ is given by
$$q_S(0)= \frac{n+2d}{n(1+2d)}\sim \frac{1}{2d}+ \frac{1}{n}$$
and, for $m\in \{1,\dots, n-1\}$ and $j\in \{1,\dots,d\}$, 
$$q_S( me_j)=\frac{2}{(1+2d)n}\sim \frac{1}{dn}.$$
This is very close to the walk that simply takes a random step in a  random coordinate and thus behaves similarly.  The mixing times for $q_S$ are different in $L^1$ and in $L^2$.
 In $L^1$ (or total variation), the mixing time is $d\log d$ (based on the coupon collector problem). In $L^2$, the mixing time is $d\log (dn)$. In both cases there is a gain over $\mu_S$ of order $n^2$.
 See \cite[page 323]{SCRW} and \cite[page 2154]{DSCcompG}.
\end{exa}

\begin{exa}[Example when hit-and-run is a little slower] Let us consider briefly the  example of random-transposition on $\mathbf S_n$. 
Because all generators have order $2$, the measure 
$q_S$ gives probability  $\frac{1}{n}+ \frac{n-1}{2n}=\frac{1}{2}(1+\frac{1}{n})$ to the identity and probability $\frac{1}{n^2}$ to any transposition. 
It follows that the hit-and-run random walk based on random transposition has a cut-off in total variation and $L^2$ at tine $n\log n$, a slow-down by a factor of $1/2$ compared to its simple random walk counterpart. 
\end{exa}

\noindent{\bf Remarks regarding hit-and-run for top-to-random}
Because the hit-and-run shuffle based on top-to-random is the focus of this paper, it is worth pointing out that it can be described naturally without reference to the general hit-and-run 
construction.  Namely, the measure $q$ at (\ref{HRTR}) can be alternatively obtained as follows:  pick a position $i$  uniformly at random in $\{1,\dots, n\}$
and then pick a packet size $j$ uniformly at random in $\{1,\dots,i\}$. Pick-up the packet of the top $j$ cards and place it below the card originally at position $i$.  This is clearly different from the top-$m$-to-random shuffles studied in \cite{DFP}.  There are two shuffles mechanisms described in \cite{DSCcompG} which bear some close similarities with the hit-and-run top-to-random shuffle described above.  They are:
\begin{itemize}
\item  The {\em crude overhand shuffle} 
\cite[page 2148]{DSCcompG}. Top, middle and bottom packets are identify using two random positions $1\le a\le b\le n$ and the order of the packets are changed as follows: top goes to bottom, middle remains in the middle, bottom goes to the top. The pair of positions $a\le b$  is chosen by  picking $a$ uniformly in $\{1,\dots, n\}$ and $b$ uniformly in $\{a,\dots, n\}$. Note that this gives weight $1/n$ to the identity which is obtained for $a=b=n$.
An $L^2$ mixing time upper-bound of order $n\log n$ is proved in \cite{DSCcompG} and a $L^1$ mixing time lower bound based on a coupon collector argument is also stated in \cite{DSCcompG}. However, although the coupon collector argument described in \cite{DSCcompG} makes heuristic sense, it seems that its detailed implementation is unclear because the probability that a pair of adjacent card be broken up depends on the position of the cards. This is worth mentioning here because the exact same difficulty appears for the hit-and-run version of top-to-random which is the focus of the present article.
\item The {\em Borel shuffle} \cite[page 2150]{DSCcompG} (which is taken from a book on the game of Bridge by Borel and Ch\'eron from 1940). In this shuffle, a middle packet is removed from the deck and placed on top. If  $(a,b)$, $1\le a\le b\le n$, describes the position of the top and bottom card of the packet removed, $(a,b)$ is picked with probability $1/{{n+1}\choose 2}$ and this gives probability $2/(n+1)$ to the identity which is obtained for any of the choices $(1,b)$, $1\le b\le n$.  An $L^2$ mixing time upper-bound of order $n\log n$ is proved in \cite{DSCcompG} as well as a $L^1$ mixing time lower bound based on a coupon collector argument (for this shuffle, the coupon collector argument works fine). 
\end{itemize}

\bibliographystyle{amsplain}

\providecommand{\bysame}{\leavevmode\hbox to3em{\hrulefill}\thinspace}
\providecommand{\MR}{\relax\ifhmode\unskip\space\fi MR }
\providecommand{\MRhref}[2]{%
  \href{http://www.ams.org/mathscinet-getitem?mr=#1}{#2}
}
\providecommand{\href}[2]{#2}

\end{document}